\documentclass[print,wide]{draft}

\usepackage{mathtools}
\usepackage{microtype}
\usepackage{enumitem}

\hypersetup{
hidelinks,
pdftitle={Null Controllability of a Stochastic Parabolic Equation with a Space-Dependent Analytic Noise Coefficient},
pdfauthor={Qi Lu and Yu Wang},
pdfkeywords={stochastic parabolic equation, null controllability, approximate controllability, state-only observability, flat torus, analytic coefficient}
}

\begin{document}

\title{Null Controllability of a Stochastic Parabolic Equation with a Space-Dependent Analytic Noise Coefficient}

\author{Qi L\"u\thanks{School of Mathematics, Sichuan University, Chengdu 610064, China. E-mail: \texttt{lu@scu.edu.cn}. Supported by the National Natural Science Foundation of China under grant 12025105.}
\and
Yu Wang\thanks{School of Mathematics, Southwest Jiaotong University, Chengdu 611756, China. E-mail: \texttt{yuwangmath@163.com}. Supported by the National Natural Science Foundation of China under grant 12401589 and by Sichuan Science and Technology Program under grant 2026NSFSC0777.}}

\date{}
\maketitle

\begin{abstract}
We consider a stochastic parabolic equation on the $n$-dimensional flat torus
$\mathbb{T}^n=(\mathbb{R}/(2\pi\mathbb Z))^n$, $n\in\mathbb N$, whose only control acts in the
drift.  The multiplicative-noise coefficient is
adapted and may depend jointly on the sample point, time, and space.  Its
spatial profiles are assumed to be real analytic with a uniform positive
radius of analyticity, while the corresponding analytic norm is only required
to be square integrable in time, uniformly along sample paths.  We prove a
state-only observability inequality with an $L^1$-in-time mean-square
observation norm and deduce null controllability from every measurable spatial
set of positive measure by one adapted drift control.  The same one-time interpolation
estimate also yields approximate controllability to arbitrary square-integrable
random terminal targets.  The proof is based on a direct analytic
energy estimate for the adjoint state, followed by propagation of smallness
and a telescoping argument.
\end{abstract}



\section{Introduction}\label{sec:introduction}


Let \(T>0\), and set  
\(\mathbb{T}^n=(\mathbb{R}/(2\pi\mathbb{Z}))^n\) \((n\in\mathbb{N})\).  
Let \((\Omega,\mathcal{F},\{\mathcal{F}_t\}_{t\in[0,T]},\mathbb{P})\) be a complete filtered probability space carrying a one-dimensional Brownian motion \(W(\cdot)\).  
We assume that \(\{\mathcal{F}_t\}_{t\in[0,T]}\) is the usual augmentation of the natural filtration of \(W(\cdot)\), and write \(\mathbb{F}\) for the progressive \(\sigma\)-algebra.  
For a separable Hilbert space \(H\), let \(L^2_{\mathcal{F}_t}(\Omega;H)\) denote the space of \(\mathcal{F}_t\)-measurable \(H\)-valued random variables \(\xi\) satisfying \(\mathbb{E}\|\xi\|_H^2<\infty\).  
The space \(L^2_{\mathbb{F}}(0,T;H)\) consists of progressively measurable \(H\)-valued processes \(X\) satisfying  
$
\mathbb{E}\int_0^T\|X(t)\|_H^2\,dt<\infty,
$
and \(L^2_{\mathbb{F}}(\Omega;C([0,T];H))\) denotes the space of adapted \(H\)-valued processes with continuous paths satisfying  
$
\mathbb{E}\sup_{0\leq t\leq T}\|X(t)\|_H^2<\infty.
$
We use \(L^\infty_{\mathbb{F}}(\Omega;L^2(0,T;H))\) for progressively measurable \(H\)-valued processes satisfying  
$
\operatorname*{ess\,sup}_{\omega\in\Omega}\int_0^T\|X(\omega,t)\|_H^2\,dt<\infty.
$
For \(r\in\{1,\infty\}\), we denote by \(L^r_{\mathbb{F}}(0,T;L^2(\Omega;H))\) the subspace of \(L^r(0,T;L^2(\Omega;H))\) consisting of equivalence classes with progressively measurable \(H\)-valued representatives, equipped with the inherited Bochner norm. Thus the time norm is taken after the \(L^2(\Omega;H)\) norm.  
For every complex Hilbert space used below, we denote its inner
product by $\langle\cdot,\cdot\rangle_H$ and take it to be linear
in the first variable. In particular,
\begin{align*}
\langle f,g\rangle_{L^2(\mathbb T^n)}
=\int_{\mathbb T^n}f(x)\overline{g(x)}\,dx.
\end{align*}
We use the same angle-bracket notation for real Hilbert spaces.

Let $G\subset\mathbb{T}^n$ be measurable with positive Lebesgue measure $|G|$.  We consider the following stochastic parabolic equation:
\begin{equation}\label{eq:forward}
\begin{cases}
\begin{aligned}
&d y=(\Delta y+\mathbf{1}_Gu)\,d t+b y\,d W(t)
&&\text{in }(0,T)\times\mathbb{T}^n,\\
&y(0)=y_0
&&\text{in }\mathbb{T}^n.
\end{aligned}
\end{cases}
\end{equation}
Here the initial datum $y_0\in L^2(\mathbb{T}^n)$, $b\in L^\infty_{\mathbb F}(\Omega;L^2(0,T;L^\infty(\mathbb{T}^n)))$ and the control
$u\in L^\infty_{\mathbb F}(0,T;L^2(\Omega;L^2(G)))$.

Classical well-posedness theory 
(e.g., \cite[Theorems~3.18 and~3.20]{Lue2021a}) ensures that \eqref{eq:forward} admits a unique mild solution
$
y\in L^2_{\mathbb F}(\Omega;C([0,T];L^2(\mathbb T^n)))
\cap L^2_{\mathbb F}(0,T;H^1(\mathbb T^n)).
$
Moreover, this solution satisfies the estimate
\begin{align}\label{eq:forward-solution-estimate}
\mathbb{E}\sup_{0\leq t\leq T}\|y(t)\|_{L^2(\mathbb T^n)}^2
+\mathbb{E}\int_0^T\|\nabla y(t)\|_{L^2(\mathbb T^n)}^2\,dt
\leq C\Bigl(\|y_0\|_{L^2(\mathbb T^n)}^2
+\mathbb{E}\int_0^T\|u(t)\|_{L^2(G)}^2\,dt\Bigr).
\end{align}
Throughout the paper,  we use $C>0$ for a constant whose value
may change from line to line, and subscripts for constants that remain fixed
within an argument.

We now recall the following definitions.

\begin{definition}[Null controllability]
Equation \eqref{eq:forward} is null controllable at time
$T$ if, for every $y_0\in L^2(\mathbb{T}^n)$, there exists
$u\in L^\infty_{\mathbb F}(0,T;L^2(\Omega;L^2(G)))$ such that
the corresponding solution to \eqref{eq:forward} satisfies
$y(T)=0$ in $L^2(\mathbb{T}^n)$, a.s.
\end{definition}

\begin{definition}[Approximate controllability]
Equation \eqref{eq:forward} is approximately controllable at time $T$ if,
for every $y_0\in L^2(\mathbb{T}^n)$,
$y_T\in L^2_{\mathcal{F}_T}(\Omega;L^2(\mathbb{T}^n))$, and $\varepsilon>0$, there exists
$u\in L^\infty_{\mathbb F}(0,T;L^2(\Omega;L^2(G)))$ such that the corresponding solution to
\eqref{eq:forward} satisfies
$\lVert y(T)-y_T\rVert_{L^2_{\mathcal{F}_T}(\Omega;L^2(\mathbb{T}^n))}<\varepsilon$.
\end{definition} 

The null controllability of deterministic parabolic equations has by now been studied extensively. A comprehensive list of references on this topic is beyond the scope of this paper; we refer interested readers to \cite{FernandezCara2012,LeRousseauLebeauRobbiano2022,Zuazua2007} and the extensive literature cited therein. By contrast, far fewer works address the controllability of stochastic parabolic equations (see \cite{Barbu2003,HernandezSantamaria2023,Lu2011,Liu2014,Liu2026,ZhangXuLiu2026,Tang2009}). The main difficulty lies in establishing the observability estimate for the dual system of a controlled stochastic parabolic equation. As pointed out in \cite{Barbu2003}, a backward stochastic parabolic equation contains both a state and a martingale integrand. The martingale integrand acts as a ``bad'' (nonhomogeneous) term when one attempts to prove the null/approximate controllability of stochastic parabolic equations via the global Carleman estimate. Consequently, two controls are typically required, one acting on the drift term and the other on the diffusion term, in order to establish the null and approximate controllability of stochastic parabolic equations (see \cite{HernandezSantamaria2023,Liu2014,ZhangXuLiu2026,Tang2009}). When $b$ is spatially independent, however, one can combine a spectral inequality with time iteration to establish the null and approximate controllability of \eqref{eq:forward} using only a drift control (see \cite{Lu2011,Yang2016a,Liu2026}).

In \cite{Lu2011,Yang2016a,Liu2026}, spatial independence plays a crucial role in these spectral arguments.   Multiplication by such a coefficient commutes with the spectral projectors of the Laplacian, so each projected adjoint equation involves only the corresponding frequencies of the martingale integrand. A space-dependent multiplier, however, generally mixes frequencies: the unobserved high-frequency part of the martingale integrand can enter the low-frequency equation for the state. Consequently, the projected equations no longer close in the form required by the scalar-coefficient argument.

In this paper, we establish null and approximate controllability for a class of spatially varying coefficients $b$. Our results allow the spatial control set to be an arbitrary measurable set of positive measure, while retaining the $L^1$-in-time observation norm used in the duality argument.  

Compared with the spatially independent single-control works \cite{Lu2011,Yang2016a,Liu2026}, we remove the restriction that the noise coefficient be spatially independent. Compared with the spatially dependent two-control works \cite{HernandezSantamaria2023,Liu2014,ZhangXuLiu2026,Tang2009}, we reduce the number of controls from two to one and obtain an \(L^1\)-in-time observability estimate from an arbitrary measurable set, at the cost of restricting the coefficient class to real-analytic functions with a uniform positive radius of analyticity. Section \ref{sec:richness} shows that this analytic class is still rich: it contains random profiles, profiles with infinitely many nonzero Fourier modes, and time-unbounded profiles whose multiplication operators generally mix Fourier modes. Thus, the present paper establishes a new trade-off between the number of controls and the regularity of the noise coefficient: within the analytic coefficient class, a single drift control suffices to overcome the frequency-mixing effect caused by spatially dependent multiplicative noise. We study the control problem on the torus \(\mathbb{T}^n\). The periodic setting plays a specific role in the proof: it makes complex translations compatible with Fourier projections and eliminates boundary contributions in the associated energy estimates. On domains with boundary, the compatibility of complex translations with the boundary conditions requires additional analysis, and the present argument does not apply directly. 

To handle the spatial dependence of the coefficients, we impose a uniform holomorphic extension condition on their spatial profiles. For each $R > 0$, we define the following complex strip and the corresponding space of holomorphic functions.

For $R>0$, set
$S_R^n=\{\zeta\in\mathbb{C}^n:|\operatorname{Im}\zeta_j|<R,\ j=1,\ldots,n\}$,
and write $\mathcal O(S_R^n)$ for the jointly holomorphic functions on $S_R^n$.
\begin{definition}[Periodic polystrip-analytic space]\label{def:analytic-space}
Let \(R>0\), and define
\[
H^\infty_{\mathrm{per}}(S_R^n)
=\Bigl\{F\in\mathcal O(S_R^n):
F(\zeta+2\pi\mathbf e_j)=F(\zeta)\ \text{for all } j=1,\dots,n,\ 
\|F\|_{H^\infty(S_R^n)}:=\sup_{\zeta\in S_R^n}|F(\zeta)|<\infty\Bigr\},
\]
where \(\mathbf e_j\) denotes the \(j\)-th standard basis vector of \(\mathbb R^n\).

Let \(\mathcal A_R(\mathbb T^n)\) denote the space of all real-valued functions on \(\mathbb T^n\) that arise as restrictions to \(\mathbb T^n\) of functions \(F\in H^\infty_{\mathrm{per}}(S_R^n)\).
\end{definition}
For \(f\in\mathcal A_R(\mathbb T^n)\), define
\begin{equation}\label{eq:analytic-space-norm}
\|f\|_{\mathcal A_R(\mathbb T^n)}
=\|F\|_{H^\infty(S_R^n)},
\end{equation}
where \(F\in H^\infty_{\mathrm{per}}(S_R^n)\) is an extension of \(f\) to \(S_R^n\). With this norm, \(\mathcal A_R(\mathbb T^n)\) is a real Banach space.

Let \(C^\omega(\mathbb T^n;\mathbb R)\) denote the space of real-valued real-analytic functions on \(\mathbb T^n\). Every such function belongs to \(\mathcal A_R(\mathbb T^n)\) for some \(R>0\), possibly depending on the function. Indeed, by local real analyticity and compactness of the torus, it admits a periodic holomorphic extension to a polystrip of positive width; after possibly reducing the width, this extension is bounded. Conversely, every function in \(\mathcal A_R(\mathbb T^n)\) is real analytic. Hence
\begin{align*}
C^\omega(\mathbb T^n;\mathbb R)
=\bigcup_{R>0}\mathcal A_R(\mathbb T^n).
\end{align*}

For the family \(b(\omega,t,\cdot)\), the following assumption imposes a common extension width and an \(L^2\)-in-time bound on the extension supremum norms that is uniform in \(\omega\).

\begin{assumption}
\label{ass:analytic-coefficient}
Let \(b:\Omega\times(0,T)\times\mathbb T^n\to\mathbb R\) be progressively measurable. We assume that there exists \(R>0\) such that, outside a fixed \((\mathbb P\otimes dt)\)-null set, \(b(\omega,t,\cdot)\in\mathcal A_R(\mathbb T^n)\).
For each such \((\omega,t)\), let \(\widetilde b(\omega,t,\cdot)\) denote the unique analytic extension of \(b(\omega,t,\cdot)\) to \(S_R^n\). We assume that these extensions admit an \(\mathbb F\otimes\mathcal B(S_R^n)\)-measurable representative
\[
\widetilde b:\Omega\times(0,T)\times S_R^n\to\mathbb C.
\]
For a.e. \((\omega,t)\), set
\begin{equation}\label{eq:beta-definition}
\beta(\omega,t):=\|b(\omega,t,\cdot)\|_{\mathcal A_R(\mathbb T^n)},
\end{equation}
and define \(\beta\) arbitrarily on the exceptional null set. We further assume that, for some \(C_1>0\),
\begin{equation}\label{eq:coefficient-budget}
\|\beta\|_{L^\infty_{\mathbb F}(\Omega;L^2(0,T))}\le C_1.
\end{equation}
\end{assumption}
\begin{remark}
Condition \eqref{eq:coefficient-budget} imposes neither time continuity nor time differentiability on \(b\), and allows \(\beta\) to be unbounded in time.
\end{remark}

The process \(\beta\) is progressively measurable.
Indeed, by \eqref{eq:analytic-space-norm}, together with continuity and periodicity, it agrees with the supremum of \(|\widetilde b|\) over a fixed countable dense subset of \([0,2\pi]^n\times(-R,R)^n\). Note also that
\begin{align}\label{eq:real-axis-bound}
|b(\omega,t,x)|\le\beta(\omega,t)
\quad\text{for almost every }(\omega,t,x).
\end{align}

Throughout the paper, constants in estimates depend only on $n$, $T$, $R$,
and $C_1$, unless otherwise stated.

Our main results are the following null and approximate controllability
theorems.

\begin{theorem}
\label{thm:null-controllability}
Assume that \Cref{ass:analytic-coefficient} holds. Then \eqref{eq:forward} is null controllable at time \(T\). Moreover, the control can be chosen to satisfy
\begin{align}\label{eq:control-cost}
\lVert u\rVert_{L^\infty_{\mathbb F}(0,T;L^2(\Omega;L^2(G)))}
\leq C_{\mathrm{obs}}\lVert y_0\rVert_{L^2(\mathbb T^n)}.
\end{align}
Here \(C_{\mathrm{obs}}>0\) is a constant that may depend in addition on \(|G|\).
\end{theorem}

\begin{theorem} 
\label{thm:approximate-controllability}
Under the assumptions of \Cref{thm:null-controllability}, equation
\eqref{eq:forward} is approximately controllable at time $T$.
\end{theorem}

By the classical duality argument, the null controllability in \Cref{thm:null-controllability} is equivalent to an observability estimate for the following adjoint equation:
\begin{equation}\label{eq:backward}
\begin{cases}
d z=-(\Delta z+bZ)\,d t+Z\,d W(t)
&\text{in }(0,T)\times\mathbb{T}^n, \\[2pt]
z(T)=\eta
&\text{in }\mathbb{T}^n. 
\end{cases}
\end{equation}
Under \Cref{ass:analytic-coefficient}, a standard adaptation of the well-posedness argument in \cite[Section~4.4]{Lue2021a} yields, for every $\eta\in L^2_{\mathcal{F}_T}(\Omega;L^2(\mathbb{T}^n))$, a unique mild solution  
$(z,Z) \in \big[L^2_{\mathbb F}(\Omega;C([0,T];L^2(\mathbb{T}^n)))
\cap L^2_{\mathbb F}(0,T;H^1(\mathbb{T}^n))\big]\times L^2_{\mathbb F}(0,T;L^2(\mathbb{T}^n))$. 
Moreover, this solution satisfies the estimate
\begin{align}\label{eq:backward-solution-estimate}
\mathbb{E}\sup_{0\leq t\leq T}\|z(t)\|_{L^2(\mathbb{T}^n)}^2
+\mathbb{E}\int_0^T\bigl(\|\nabla z(t)\|_{L^2(\mathbb{T}^n)}^2+\|Z(t)\|_{L^2(\mathbb{T}^n)}^2\bigr)\,d t
\leq C\mathbb{E}\|\eta\|_{L^2(\mathbb{T}^n)}^2.
\end{align}
Our main observability result is the following.

\begin{theorem}
\label{thm:state-observability}
Assume that \Cref{ass:analytic-coefficient} holds. 
Then, every solution
$(z,Z)$ of \eqref{eq:backward} satisfies
\begin{align}\label{eq:state-observability}
\lVert z(0)\rVert_{L^2(\mathbb{T}^n)}
\leq C_{\mathrm{obs}}
\lVert z\rVert_{L^1_{\mathbb F}(0,T;L^2(\Omega;L^2(G)))},
\end{align}
with $C_{\mathrm{obs}}$ as in \eqref{eq:control-cost}.
\end{theorem}
The proof of \Cref{thm:state-observability} rests on a direct analytic energy estimate for the adjoint state. On each interval \([t,s]\), we apply a time-dependent complex translation to the Fourier--Galerkin equations. Since this translation vanishes at the terminal time \(s\), no analytic regularity is required of the terminal datum. The positive It\^o correction involving the translated martingale integrand absorbs the translated coupling term \(bZ\), while the Laplacian controls the term generated by the time derivative of the translation. The analytic bound on \(b\), together with \eqref{eq:coefficient-budget}, makes the resulting estimates uniform in the Galerkin dimension. Passing to the limit yields quantitative spatial analyticity of \(z(t)\) without imposing spatial regularity on \(Z\).

We then apply propagation of smallness to the \(L^2(\Omega)\)-valued analytic state and obtain a one-time interpolation inequality from every measurable set \(G\) of positive measure. Combining this inequality with one-sided energy propagation and geometric telescoping yields the \(L^1\)-in-time estimate in \Cref{thm:state-observability}. Its \(L^1\)--\(L^\infty\) duality yields \Cref{thm:null-controllability}, and the same interpolation inequality supplies the unique continuation needed for \Cref{thm:approximate-controllability}.

The remainder of the paper is organized as follows. Section~\ref{sec:analytic-smoothing} proves the analytic energy estimate. Section~\ref{sec:propagation} derives the one-time interpolation and observability. Section~\ref{sec:duality} proves null and approximate controllability by duality. Section~\ref{sec:richness} illustrates the richness of the analytic coefficient class. Appendix~\ref{app:galerkin} contains the proof of \Cref{lem:galerkin-identification}.

\section{Analytic energy estimates for the adjoint state}\label{sec:analytic-smoothing}

Throughout Sections~\ref{sec:analytic-smoothing}--\ref{sec:duality}, we assume that the coefficient \(b\) satisfies \Cref{ass:analytic-coefficient} with fixed \(R\) and \(C_1\), and we let \((z,Z)\) denote a solution of \eqref{eq:backward}. We work in the complexifications of the periodic Hilbert spaces.

In this section, we prove the following analytic smoothing estimate for \(z\), without assuming spatial regularity of \(Z\). This estimate will be used in Section~\ref{sec:propagation} to derive a one-time interpolation inequality and an observability estimate.

\begin{theorem}[Quantitative analytic smoothing]
\label{thm:analytic-smoothing}
For every $0\leq t<T$, the state $z(t)$ has an
$L^2_{\mathcal{F}_t}(\Omega;\mathbb{C})$-valued holomorphic extension
$\mathscr Z_t$ to $S_{3R/4}^n$, periodic in each coordinate.
There exists $C>0$, depending only on $n$ and $R$, such that, for every
$0\leq t<s\leq T$,
\begin{align}\label{eq:fixed-strip-smoothing}
\sup_{\zeta\in\overline{S_{R/2}^n}}
\lVert \mathscr Z_t(\zeta)\rVert_{L^2(\Omega)}
\leq C e^{C_1^2/2}
\exp\biggl(\frac{9nR^2}{64(s-t)}\biggr)
\lVert z(s)\rVert_{L^2_{\mathcal{F}_s}(\Omega;L^2(\mathbb{T}^n))}.
\end{align}
\end{theorem}

\subsection{Galerkin approximation and analytic energy estimates}

To prove \Cref{thm:analytic-smoothing}, we estimate the Galerkin solutions
by a complex translation and then pass to the limit.

Let $e_{\mathbf k}(x)=(2\pi)^{-n/2}e^{\mathrm{i}{\mathbf k}\cdot x}$, $\mathbf k=(k_1,\cdots,k_n)\in\mathbb Z^n$,
be the orthonormal Fourier basis of $L^2(\mathbb{T}^n)$.  Set
$D_j=-\mathrm{i}\partial_{x_j}$ and
$D=(D_1,\ldots,D_n)$. 
Let $P_N$ be the orthogonal projection onto
$H_N:=\operatorname{span}\{e_{\mathbf k}:|\mathbf k|_\infty\leq N\}$ with $|\mathbf k|_\infty = \max\limits_{1\leq j\leq n}|k_j|$.
Fix $0\leq t<s\leq T$ and set $\xi=z(s)$.
On $[0,s]$, approximate $(z,Z)$ by the $H_N$-valued solution $(z_N,Z_N)$ of
\begin{equation}\label{eq:galerkin-bsde}
\begin{cases}
\begin{aligned}
&d z_N
=-\bigl(\Delta z_N+P_N(bZ_N)\bigr)\,dr
+Z_N\,dW(r)
&&\text{in }(0,s),\\
&z_N(s)=P_N\xi.
\end{aligned}
\end{cases}
\end{equation}
Since $P_N$ is an orthogonal projection, \eqref{eq:real-axis-bound} gives
\begin{align*}
\|P_N(bv)\|_{L^2(\mathbb{T}^n)}
\leq\|bv\|_{L^2(\mathbb{T}^n)}
\leq\beta\|v\|_{L^2(\mathbb{T}^n)},\qquad v\in H_N.
\end{align*}
In view of \eqref{eq:coefficient-budget}, a standard adaptation of the
well-posedness argument in \cite[Section~4.4]{Lue2021a} yields a unique
solution of \eqref{eq:galerkin-bsde}.  The following convergence result
is proved in Appendix~\ref{app:galerkin}.

\begin{lemma}[Galerkin convergence]\label{lem:galerkin-identification}
For every $r\in[0,s]$, it holds that
\begin{align}\label{eq:galerkin-strong-time}
\lim_{N\to\infty}\|z_N(r)- z(r)\|_{L^2_{\mathcal{F}_r}(\Omega;L^2(\mathbb{T}^n))}
=0 
\end{align}
and
\begin{align}\label{eq:galerkin-strong-time-1}
\begin{cases}
\displaystyle\lim_{N\to\infty}\|z_N- z\|_{L^2_{\mathbb F}(\Omega;C([0,s];L^2(\mathbb{T}^n)))
\cap L^2_{\mathbb F}(0,s;H^1(\mathbb{T}^n))}=0,\\[2pt]
\displaystyle\lim_{N\to\infty}\|Z_N- Z\|_{L^2_{\mathbb F}(0,s;L^2(\mathbb{T}^n))}=0.
\end{cases}
\end{align}
\end{lemma}

For \(v\in L^2(\mathbb T^n)\), we define its Fourier coefficients by
\begin{align*}
\widehat v(\mathbf k)=\frac1{(2\pi)^n}\int_{\mathbb T^n}v(x)e^{-\mathrm{i}\mathbf k\cdot x}\,dx,
\qquad \mathbf k\in\mathbb Z^n.
\end{align*}

For \(a\in\mathbb R\) and
\(\varepsilon=(\varepsilon_1,\ldots,\varepsilon_n)\in\{+1,-1\}^n\), the
Fourier multiplier \(e^{a\varepsilon\cdot D}\) is defined on \(H_N\) by
\(e^{a\varepsilon\cdot D}e_{\mathbf k}=e^{a\varepsilon\cdot \mathbf k}e_{\mathbf k}\) and satisfies
\begin{align}\label{eq:complex-translation}
(e^{a\varepsilon\cdot D}v)(x)=v(x-\mathrm{i} a\varepsilon),
\qquad v\in H_N.
\end{align}
We work outside the fixed null set appearing in \Cref{ass:analytic-coefficient}.
Here \(\widetilde b\) denotes the analytic extension from that assumption.
For every \(v\in H_N\), \(|a|<R\), and
\(\varepsilon\in\{+1,-1\}^n\), we also have
\begin{align}\label{eq:projected-shift-identity}
e^{a\varepsilon\cdot D}P_N(bv)
=P_N\bigl[\widetilde b(\,\cdot-\mathrm{i} a\varepsilon\,)
e^{a\varepsilon\cdot D}v\bigr].
\end{align}
Indeed, Cauchy's theorem applied in each coordinate gives
\begin{align*}
\tikz[baseline=(product.base)]{
\node[inner sep=0pt,outer sep=0pt] (product)
{$\widetilde b(\,\cdot-\mathrm{i} a\varepsilon\,)
v(\,\cdot-\mathrm{i} a\varepsilon\,)$};
\draw[line width=0.35pt]
([yshift=1.5pt]product.north west) --
([yshift=4pt]product.north) --
([yshift=1.5pt]product.north east);
}(\mathbf k)
=e^{a\varepsilon\cdot \mathbf k}\widehat{bv}(\mathbf k).
\end{align*}
The contours lie in a smaller polystrip, and the integrals over the
vertical sides cancel by periodicity. By \eqref{eq:complex-translation},
retaining the Fourier modes with \(|\mathbf k|_\infty\le N\) yields
\eqref{eq:projected-shift-identity}.

To combine the estimates over all sign vectors, we use the Fourier multiplier
\(e^{\rho|D|_1}\), defined for \(\rho>0\) by
\begin{align*}
(e^{\rho|D|_1}v)(x)
=\sum_{\mathbf k\in\mathbb Z^n}e^{\rho|\mathbf k|_1}\widehat v(\mathbf k) e^{i\mathbf k\cdot x},
\end{align*}
where \(|\mathbf k|_1=\sum\limits_{j=1}^n|k_j|\).
Its domain consists of those \(v\in L^2(\mathbb T^n)\) for which this series
converges in \(L^2(\mathbb T^n)\), equivalently,
\begin{align*}
\sum_{\mathbf k\in\mathbb Z^n}e^{2\rho|\mathbf k|_1}|\widehat v(\mathbf k)|^2<\infty.
\end{align*}

\begin{proposition}[Finite-dimensional analytic energy estimate]
\label{prop:finite-analytic-energy}
Let $0<\rho<R$.  For every $N\in\mathbb{N}$ and
$\varepsilon\in\{+1,-1\}^n$,
\begin{align}\label{eq:one-sided-shift-estimate}
\mathbb{E}\bigl\lVert e^{\rho\varepsilon\cdot D}z_N(t)\bigr\rVert_{L^2(\mathbb{T}^n)}^2
\leq
\exp\biggl(\frac{n\rho^2}{2(s-t)}+C_1^2\biggr)
\mathbb{E}\lVert P_N z(s)\rVert_{L^2(\mathbb{T}^n)}^2.
\end{align}
Consequently,
\begin{align}\label{eq:finite-analytic-norm-estimate}
\mathbb{E}\bigl\lVert e^{\rho|D|_1}z_N(t)\bigr\rVert_{L^2(\mathbb{T}^n)}^2
\leq
2^n\exp\biggl(\frac{n\rho^2}{2(s-t)}+C_1^2\biggr)
\mathbb{E}\lVert z(s)\rVert_{L^2(\mathbb{T}^n)}^2.
\end{align}
\end{proposition}
\begin{proof} We divide the proof into two steps.

\textbf{Step 1.} We first prove the estimate for a fixed sign vector.
Fix \(N\in\mathbb N\) and \(\varepsilon\in\{+1,-1\}^n\).
Set \(\delta=s-t\) and \(\sigma(r)=\rho(s-r)/\delta\), \(r\in[t,s]\).
Then \(\sigma(t)=\rho\), \(\sigma(s)=0\), and \(\sigma'=-\rho/\delta\).
Define
\begin{align}\label{eq:shifted-galerkin-variables}
X_N^\varepsilon(r)=e^{\sigma(r)\varepsilon\cdot D}z_N(r),
\qquad
Y_N^\varepsilon(r)=e^{\sigma(r)\varepsilon\cdot D}Z_N(r), \qquad b_\varepsilon(r,x)=\widetilde b(r,x-\mathrm{i}\sigma(r)\varepsilon).
\end{align}
Since \(0\leq\sigma(r)\leq\rho<R\),
\eqref{eq:analytic-space-norm} gives
\begin{align}\label{eq:shifted-coefficient-bound}
|b_\varepsilon(\omega,r,x)|\leq\beta(\omega,r).
\end{align}
By \eqref{eq:shifted-galerkin-variables} and
\eqref{eq:projected-shift-identity}, we have
\begin{align*}
e^{\sigma\varepsilon\cdot D}P_N(bZ_N)=P_N(b_\varepsilon Y_N^\varepsilon).
\end{align*}
For fixed $N$, the multiplier $e^{\sigma(r)\varepsilon\cdot D}$ is continuously differentiable in $r$ on $H_N$.
Since \(D_j\), \(P_N\), and \(e^{\sigma\varepsilon\cdot D}\) commute,
It\^o's formula gives
\begin{align}\label{eq:shifted-bsde}
dX_N^\varepsilon
=\bigl[-\Delta X_N^\varepsilon
+\sigma'\varepsilon\cdot DX_N^\varepsilon
-P_N(b_\varepsilon Y_N^\varepsilon)\bigr]\,dr
+Y_N^\varepsilon\,dW(r).
\end{align}
Since \(X_N^\varepsilon\in H_N\) and \(P_N\) is an orthogonal projection,
\begin{align*}
\bigl\langle X_N^\varepsilon,P_N(b_\varepsilon Y_N^\varepsilon)\bigr\rangle_{L^2(\mathbb{T}^n)}
&=\bigl\langle P_N X_N^\varepsilon,b_\varepsilon Y_N^\varepsilon\bigr\rangle_{L^2(\mathbb{T}^n)}
=\bigl\langle X_N^\varepsilon,b_\varepsilon Y_N^\varepsilon\bigr\rangle_{L^2(\mathbb{T}^n)}.
\end{align*}
Applying It\^o's formula to \(\lVert X_N^\varepsilon\rVert_{L^2(\mathbb{T}^n)}^2\)
in \eqref{eq:shifted-bsde}, we obtain
\begin{align}\label{9.13-eq1}
d\lVert X_N^\varepsilon\rVert_{L^2(\mathbb{T}^n)}^2
&=\Big[
2\sum_{j=1}^n\lVert D_jX_N^\varepsilon\rVert_{L^2(\mathbb{T}^n)}^2
+2\sigma'\sum_{j=1}^n\varepsilon_j
\operatorname{Re}\langle D_jX_N^\varepsilon,X_N^\varepsilon\rangle_{L^2(\mathbb{T}^n)}\notag\\
&\qquad
+\lVert Y_N^\varepsilon\rVert_{L^2(\mathbb{T}^n)}^2
-2\operatorname{Re}\bigl\langle X_N^\varepsilon,b_\varepsilon Y_N^\varepsilon\bigr\rangle_{L^2(\mathbb{T}^n)}
\Big]\,dr
+2\operatorname{Re}\langle X_N^\varepsilon,Y_N^\varepsilon\rangle_{L^2(\mathbb{T}^n)}\,dW(r).
\end{align}
The first two terms on the right-hand side of \eqref{9.13-eq1} can be rewritten as
\begin{align}
&2\sum_{j=1}^n\lVert D_jX_N^\varepsilon\rVert_{L^2(\mathbb{T}^n)}^2
+2\sigma'\sum_{j=1}^n\varepsilon_j
\operatorname{Re}\langle D_jX_N^\varepsilon,X_N^\varepsilon\rangle_{L^2(\mathbb{T}^n)}\notag\\
= &2\sum_{j=1}^n
\biggl\lVert \biggl(D_j+\frac{\sigma'\varepsilon_j}2\biggr)
X_N^\varepsilon\biggr\rVert_{L^2(\mathbb{T}^n)}^2
-\frac{n\rho^2}{2\delta^2}\lVert X_N^\varepsilon\rVert_{L^2(\mathbb{T}^n)}^2,
\label{eq:complete-square-D}
\end{align}
and the third and fourth terms can be rewritten as
\begin{align}
\lVert Y_N^\varepsilon\rVert_{L^2(\mathbb{T}^n)}^2
-2\operatorname{Re}\bigl\langle X_N^\varepsilon,b_\varepsilon Y_N^\varepsilon\bigr\rangle_{L^2(\mathbb{T}^n)}
={}&\bigl\lVert Y_N^\varepsilon-
\overline{b_\varepsilon}X_N^\varepsilon\bigr\rVert_{L^2(\mathbb{T}^n)}^2
-\lVert b_\varepsilon X_N^\varepsilon\rVert_{L^2(\mathbb{T}^n)}^2.
\label{eq:complete-square-q}
\end{align}
To compensate for the two negative terms in \eqref{eq:complete-square-D} and \eqref{eq:complete-square-q}, set
\begin{align*}
\Gamma(r)=\exp\biggl(\int_t^r
\biggl[\frac{n\rho^2}{2\delta^2}+\beta(\omega,\tau)^2\biggr]
\,d\tau\biggr).
\end{align*}
Since \(d\Gamma=(n\rho^2/(2\delta^2)+\beta^2)\Gamma\,dr\), the product rule
together with \eqref{eq:complete-square-D}--\eqref{eq:complete-square-q} yields
\begin{align}
&d\bigl(\Gamma\lVert X_N^\varepsilon\rVert_{L^2(\mathbb{T}^n)}^2\bigr)\notag
\\ &= \Gamma\,d\lVert X_N^\varepsilon\rVert_{L^2(\mathbb{T}^n)}^2
+\lVert X_N^\varepsilon\rVert_{L^2(\mathbb{T}^n)}^2\,d\Gamma 
\label{eq:weighted-positive-identity}
\\
&= \Gamma\biggl[
2\sum_{j=1}^n\biggl\lVert \biggl(D_j+\frac{\sigma'\varepsilon_j}2\biggr)
X_N^\varepsilon\biggr\rVert_{L^2(\mathbb{T}^n)}^2
+\bigl\lVert Y_N^\varepsilon-
\overline{b_\varepsilon}X_N^\varepsilon\bigr\rVert_{L^2(\mathbb{T}^n)}^2
+\beta^2\lVert X_N^\varepsilon\rVert_{L^2(\mathbb{T}^n)}^2 
-\lVert b_\varepsilon X_N^\varepsilon\rVert_{L^2(\mathbb{T}^n)}^2
\biggr]\,dr+\,dM_r,\notag
\end{align}
where the stochastic integral is given by
\begin{align*}
dM_r=2\Gamma(r)\operatorname{Re}\langle X_N^\varepsilon(r),Y_N^\varepsilon(r)\rangle_{L^2(\mathbb{T}^n)}\,dW(r),
\qquad M_t=0.
\end{align*}
By \eqref{eq:coefficient-budget}, $1\leq\Gamma(r)\leq\exp(n\rho^2/(2\delta)+C_1^2)$.
Since the Galerkin solution satisfies
$z_N\in L^2_{\mathbb F}(\Omega;C([0,s];H_N))$ and $Z_N\in L^2_{\mathbb F}(0,s;H_N)$,
the bound on $\Gamma$ and the boundedness of $e^{\sigma(r)\varepsilon\cdot D}$ on $H_N$ yield
\begin{align*}
\mathbb E\biggl(\int_t^s
\bigl|2\Gamma(r)\operatorname{Re}\langle X_N^\varepsilon(r),Y_N^\varepsilon(r)\rangle_{L^2(\mathbb T^n)}\bigr|^2\,dr\biggr)^{1/2}<\infty.
\end{align*}
Hence the stochastic integral has zero expectation by
\cite[Theorem~2.130, item~3)]{Lue2021a} with $p=1$.
The drift term in \eqref{eq:weighted-positive-identity} is nonnegative by
\eqref{eq:shifted-coefficient-bound}. Integrating from \(t\) to \(s\) and
taking expectations, using \(\Gamma(t)=1\), \(\sigma(t)=\rho\), and \(\sigma(s)=0\),
we obtain
\begin{align*}
\mathbb{E}\bigl\lVert e^{\rho\varepsilon\cdot D}z_N(t)\bigr\rVert_{L^2(\mathbb{T}^n)}^2
&\leq\mathbb{E}\bigl[\Gamma(s)\lVert P_N z(s)\rVert_{L^2(\mathbb{T}^n)}^2\bigr] \leq
\exp\biggl(\frac{n\rho^2}{2\delta}+C_1^2\biggr)
\mathbb{E}\lVert P_N z(s)\rVert_{L^2(\mathbb{T}^n)}^2,
\end{align*}
which proves \eqref{eq:one-sided-shift-estimate}.

\textbf{Step 2.} We now combine the estimates over all sign vectors.
For every \(\mathbf k\in\mathbb Z^n\), choose a sign vector \(\varepsilon\) such that
\(\varepsilon\cdot \mathbf k=|\mathbf k|_1\). Hence
\begin{align*}
e^{2\rho|\mathbf k|_1}
\leq\sum_{\varepsilon\in\{+1,-1\}^n}e^{2\rho\varepsilon\cdot \mathbf k}.
\end{align*}
Parseval's identity then gives
\begin{align*}
\bigl\lVert e^{\rho|D|_1}z_N(t)\bigr\rVert_{L^2(\mathbb{T}^n)}^2
\leq
\sum_{\varepsilon\in\{+1,-1\}^n}
\bigl\lVert e^{\rho\varepsilon\cdot D}z_N(t)\bigr\rVert_{L^2(\mathbb{T}^n)}^2.
\end{align*}
Applying \eqref{eq:one-sided-shift-estimate} to each term and using
\(\|P_Nz(s)\|_{L^2(\mathbb{T}^n)}\leq\|z(s)\|_{L^2(\mathbb{T}^n)}\) proves
\eqref{eq:finite-analytic-norm-estimate}.
\end{proof}

\subsection{Proof of the analytic smoothing estimate}

\begin{proof}[Proof of \Cref{thm:analytic-smoothing}]
We begin by passing to the limit in the  estimate \eqref{eq:finite-analytic-norm-estimate}. Fix \(0\leq t<s\leq T\), \(0<\rho<R\), and \(K\in\mathbb N\).
By \Cref{lem:galerkin-identification}, \(P_K z_N(t)\to P_K z(t)\) strongly in
\(L^2_{\mathcal F_t}(\Omega;L^2(\mathbb T^n))\). Since \(e^{\rho|D|_1}\) is bounded on \(H_K\), \eqref{eq:finite-analytic-norm-estimate} gives
\begin{align*}
\mathbb{E}\bigl\lVert e^{\rho|D|_1}P_K z(t)\bigr\rVert_{L^2(\mathbb T^n)}^2
&=\lim_{N\to\infty}
\mathbb{E}\bigl\lVert e^{\rho|D|_1}P_K z_N(t)\bigr\rVert_{L^2(\mathbb T^n)}^2\\
&\leq
2^n\exp\biggl(\frac{n\rho^2}{2(s-t)}+C_1^2\biggr)
\mathbb{E}\lVert z(s)\rVert_{L^2(\mathbb T^n)}^2.
\end{align*}
Letting \(K\to\infty\) and using monotone convergence of the Fourier series yields
\begin{equation}
	\label{eq:infinite-analytic-norm-estimate}
	\mathbb{E}\bigl\lVert e^{\rho|D|_1}z(t)\bigr\rVert_{L^2(\mathbb T^n)}^2
	\leq
	2^n\exp\biggl(\frac{n\rho^2}{2(s-t)}+C_1^2\biggr)
	\mathbb{E}\lVert z(s)\rVert_{L^2(\mathbb T^n)}^2.
\end{equation}

To construct the extension, set \(z_{\mathbf k}(t)=\widehat{z(t)}(\mathbf k)\) and write
\begin{align*}
z(t,x)=\sum_{\mathbf k\in\mathbb Z^n}z_{\mathbf k}(t)e^{\mathrm{i} \mathbf k\cdot x},
\qquad z_{\mathbf k}(t)\in L^2_{\mathcal{F}_t}(\Omega;\mathbb C).
\end{align*}
For \(\zeta\in\overline{S_r^n}\) with \(0<r<\rho\), the Cauchy--Schwarz inequality gives
\begin{align}\label{9.13-eq2}
\sum_{\mathbf k\in\mathbb Z^n}
\lVert z_{\mathbf k}(t)e^{\mathrm{i} \mathbf k\cdot \zeta}\rVert_{L^2(\Omega)}
&\leq
\Bigl(\sum_{\mathbf k\in\mathbb Z^n}e^{-2(\rho-r)|\mathbf k|_1}\Bigr)^{1/2}
\Bigl(\sum_{\mathbf k\in\mathbb Z^n}e^{2\rho|\mathbf k|_1}
\lVert z_{\mathbf k}(t)\rVert_{L^2(\Omega)}^2\Bigr)^{1/2}.
\end{align}
The first factor on the right-hand side of \eqref{9.13-eq2} is finite,
while the second is bounded by \eqref{eq:infinite-analytic-norm-estimate}. Hence the Fourier series
converges absolutely and uniformly in \(L^2_{\mathcal{F}_t}(\Omega;\mathbb C)\)
on every smaller closed polystrip. Its sum \(\mathscr Z_t\) is holomorphic on \(S_\rho^n\),
periodic, and agrees with \(z(t)\) on the real torus $\mathbb T^n$. The same estimate yields
\begin{align}\label{eq:general-strip-bound}
\sup_{\zeta\in\overline{S_r^n}}
\lVert \mathscr Z_t(\zeta)\rVert_{L^2(\Omega)}
\leq
C
\exp\biggl(\frac{n\rho^2}{4(s-t)}+\frac{C_1^2}{2}\biggr)
\lVert z(s)\rVert_{L^2_{\mathcal{F}_s}(\Omega;L^2(\mathbb T^n))}.
\end{align}
Here \(C\) depends only on \(n\) and \(\rho-r\).
Taking \(\rho=3R/4\) and \(r=R/2\) proves \eqref{eq:fixed-strip-smoothing}.
\end{proof}

\section{Propagation of smallness and observability}\label{sec:propagation}

We combine analytic propagation of smallness with
\eqref{eq:fixed-strip-smoothing} and then apply a telescoping argument.

\subsection{Analytic propagation and one-time interpolation}

For a multi-index $\alpha=(\alpha_1,\ldots,\alpha_n)\in\mathbb N_0^n$,
where $\mathbb N_0=\{0,1,\ldots\}$, we use
\begin{align*}
|\alpha|=\sum_{j=1}^n\alpha_j,\qquad
\alpha!=\prod_{j=1}^n\alpha_j!,\qquad
\partial^\alpha=\partial_{x_1}^{\alpha_1}\cdots\partial_{x_n}^{\alpha_n}.
\end{align*}
We first recall \cite[Theorem~4]{Apraiz2014}.
\begin{lemma}[Multidimensional analytic propagation]
\label{thm:scalar-propagation}
Let \(\ell>0\), let \(E\subset B_\ell(x_0)\subset\mathbb R^n\) be measurable with positive Lebesgue measure \(|E|\), and let \(f:B_{2\ell}(x_0)\to\mathbb R\) be real-analytic. Assume that, for some \(\mathcal M>0\) and \(0<\varrho\le1\),
\[
|\partial^\alpha f(x)|
\le \mathcal M\frac{|\alpha|!}{(\varrho\ell)^{|\alpha|}},
\qquad x\in B_{2\ell}(x_0),\quad \alpha\in\mathbb N_0^n.
\]
Then there exist constants \(N\ge1\) and \(\vartheta\in(0,1)\), depending only on \(n\), \(\varrho\), and \(|E|/|B_\ell(x_0)|\), such that
\[
\|f\|_{L^\infty(B_\ell(x_0))}
\le N
\left(\frac1{|E|}\int_E|f(x)|\,dx\right)^\vartheta
\mathcal M^{1-\vartheta}.
\]
\end{lemma}

We next extend this estimate to periodic Hilbert space-valued functions.

\begin{proposition}[Periodic Hilbert space-valued propagation]
\label{prop:hilbert-propagation}
Let \(H\) be a separable complex Hilbert space, let \(r>0\), and let \(E\subset\mathbb T^n\) be a measurable set of positive measure. Then there exist constants \(C>0\) and \(\theta\in(0,1)\), depending only on \(n\), \(r\), and \(|E|\), such that the following holds. If \(f:S_r^n\to H\) is holomorphic, \(2\pi\)-periodic in each coordinate, and satisfies
\[
\sup_{\zeta\in S_r^n}\|f(\zeta)\|_H\le\mathcal M
\]
for some \(\mathcal M\ge0\), then
\begin{equation}\label{eq:hilbert-propagation}
\|f\|_{L^2(\mathbb T^n;H)}
\le C
\|f\|_{L^2(E;H)}^\theta
\mathcal M^{1-\theta}.
\end{equation}
In particular, the constants \(C\) and \(\theta\) do not depend on \(H\).
\end{proposition}

\begin{proof}
If \(\mathcal M=0\) or \(H=\{0\}\), the conclusion is immediate. Hence we may assume that \(\mathcal M>0\) and \(H\ne\{0\}\). The proof is divided into two steps.

\smallskip

\textbf{Step 1.} We first apply the scalar propagation estimate.
Regard \(E\) as a subset of \(Q=[-\pi,\pi)^n\), and set \(\varrho=\min\{1,r/(4\pi\sqrt n)\}\). Then
\(E\subset Q\subset B_{2\pi\sqrt n}(0)\) and \(2\pi\sqrt n\varrho \le r/2\).
For each unit vector \(h\in H\) and each \(\gamma\in\mathbb R\), define
\begin{align*}
g_{h,\gamma}(x)
=\operatorname{Re}\bigl(e^{-\mathrm{i}\gamma}\langle f(x),h\rangle_H\bigr),
\qquad x\in\mathbb R^n.
\end{align*}
Since the inner product is linear in its first variable, the function
\(\zeta\mapsto\langle f(\zeta),h\rangle_H\) is holomorphic on \(S_r^n\) and bounded by \(\mathcal M\).
Every polydisc of radius \(r/2\) centered at a real point is contained in \(S_r^n\).
Thus Cauchy's inequality yields, for \(x\in B_{4\pi\sqrt n}(0)\),
\begin{align*}
|\partial^\alpha g_{h,\gamma}(x)|
&\leq |\langle \partial^\alpha f(x),h\rangle_H|
\leq \mathcal M\alpha!\biggl(\frac2r\biggr)^{|\alpha|}
\leq \mathcal M\frac{|\alpha|!}{(2\pi\sqrt n\varrho )^{|\alpha|}}.
\end{align*}
Here we used \(\alpha!\le|\alpha|!\) and \(2\pi\sqrt n\varrho \le r/2\).
Therefore \Cref{thm:scalar-propagation} applies and gives
\begin{align*}
\sup_{x\in B_{2\pi\sqrt n}(0)}|g_{h,\gamma}(x)|
\leq C\biggl(\frac1{|E|}\int_E|g_{h,\gamma}(x)|\,dx\biggr)^\theta
\mathcal M^{1-\theta}.
\end{align*}
For the observation term, we have
\begin{align*}
\frac1{|E|}\int_E|g_{h,\gamma}(x)|\,dx
&\leq\frac1{|E|}\int_E\|f(x)\|_H\,dx
\leq |E|^{-1/2}\|f\|_{L^2(E;H)}.
\end{align*}
Consequently,
\begin{align*}
\sup_{x\in B_{2\pi\sqrt n}(0)}|g_{h,\gamma}(x)|
\leq C|E|^{-\theta/2}
\|f\|_{L^2(E;H)}^\theta\mathcal M^{1-\theta},
\end{align*}
where \(C\) and \(\theta\) depend only on \(n\), \(r\), and \(|E|\), and are independent of \(h\), \(\gamma\), and \(H\).

\smallskip

\textbf{Step 2.} We now recover the Hilbert-space norm.
For every \(x\in Q\), the dual characterization of the norm yields
\begin{align*}
\|f(x)\|_H
&=\sup_{\|h\|_H=1}|\langle f(x),h\rangle_H|
=\sup_{\|h\|_H=1,\,\gamma\in\mathbb R}
\bigl|\operatorname{Re}\bigl(e^{-\mathrm{i}\gamma}\langle f(x),h\rangle_H\bigr)\bigr|.
\end{align*}
Taking the supremum over \(h\) and \(\gamma\) in the estimate from Step~1, we obtain
\begin{align*}
\sup_{x\in Q}\|f(x)\|_H
\leq C|E|^{-\theta/2}
\|f\|_{L^2(E;H)}^\theta\mathcal M^{1-\theta}.
\end{align*}
Since \(f\) is periodic, integrating over \(Q\) gives
\begin{align*}
\|f\|_{L^2(\mathbb T^n;H)}
&=\biggl(\int_Q\|f(x)\|_H^2\,dx\biggr)^{1/2}\leq |Q|^{1/2}\sup_{x\in Q}\|f(x)\|_H
\leq C\|f\|_{L^2(E;H)}^\theta\mathcal M^{1-\theta}.
\end{align*}
This proves \eqref{eq:hilbert-propagation}, with constants independent of \(H\).
\end{proof}

Combining \eqref{eq:hilbert-propagation} with \eqref{eq:fixed-strip-smoothing}
gives the following interpolation estimate.
\begin{corollary}[One-time interpolation]\label{cor:one-time-interpolation}
Let \((z,Z)\) solve \eqref{eq:backward}. Then there exist constants \(C>0\) and \(\theta\in(0,1)\), where \(C\) depends only on \(n,R,C_1,|G|\) and \(\theta\) depends only on \(n,R,|G|\), such that for all \(0\le t<s\le T\),
\begin{align}\label{eq:one-time-multiplicative}
\lVert z(t)\rVert_{L^2_{\mathcal{F}_t}(\Omega;L^2(\mathbb{T}^n))}
&\leq C\exp\biggl(\frac{C}{s-t}\biggr)
\lVert z(t)\rVert_{L^2_{\mathcal{F}_t}(\Omega;L^2(G))}^\theta
\lVert z(s)\rVert_{L^2_{\mathcal{F}_s}(\Omega;L^2(\mathbb{T}^n))}^{1-\theta}.
\end{align}
Equivalently, for every \(\varepsilon>0\),
\begin{align}\label{eq:one-time-epsilon}
\lVert z(t)\rVert_{L^2_{\mathcal{F}_t}(\Omega;L^2(\mathbb{T}^n))}
&\leq \varepsilon
\lVert z(s)\rVert_{L^2_{\mathcal{F}_s}(\Omega;L^2(\mathbb{T}^n))}
+C\varepsilon^{-\frac{1-\theta}{\theta}}
\exp\biggl(\frac{C}{s-t}\biggr)
\lVert z(t)\rVert_{L^2_{\mathcal{F}_t}(\Omega;L^2(G))}.
\end{align}
\end{corollary}
\begin{proof}
Fix \(0\le t<s\le T\) and set
\(H=L^2_{\mathcal F_t}(\Omega;\mathbb C)\).
By \Cref{thm:analytic-smoothing}, \(\mathscr Z_t\) is an \(H\)-valued
holomorphic extension of \(z(t)\) to \(S_{3R/4}^n\).
Since it agrees with \(z(t)\) on the real torus \(\mathbb T^n\), Fubini's theorem yields
\begin{align*}
\|\mathscr Z_t\|_{L^2(\mathbb T^n;H)}^2
&=\int_{\mathbb T^n}\mathbb E|z(t,x)|^2\,dx
=\|z(t)\|_{L^2_{\mathcal F_t}(\Omega;L^2(\mathbb T^n))}^2,\\
\|\mathscr Z_t\|_{L^2(G;H)}^2
&=\int_G\mathbb E|z(t,x)|^2\,dx
=\|z(t)\|_{L^2_{\mathcal F_t}(\Omega;L^2(G))}^2.
\end{align*}
Applying \Cref{prop:hilbert-propagation} with polystrip width \(R/2\)
and observation set \(G\), and using the preceding identities together with
\eqref{eq:fixed-strip-smoothing}, we obtain
\begin{align*}
&\|z(t)\|_{L^2_{\mathcal F_t}(\Omega;L^2(\mathbb T^n))}\\
&\leq C\|z(t)\|_{L^2_{\mathcal F_t}(\Omega;L^2(G))}^\theta
\Bigl(\sup_{\zeta\in S_{R/2}^n}\|\mathscr Z_t(\zeta)\|_{L^2(\Omega)}\Bigr)^{1-\theta}\\
&\leq C e^{(1-\theta)C_1^2/2}
\exp\biggl(\frac{9nR^2(1-\theta)}{64(s-t)}\biggr)
\|z(t)\|_{L^2_{\mathcal F_t}(\Omega;L^2(G))}^\theta
\|z(s)\|_{L^2_{\mathcal F_s}(\Omega;L^2(\mathbb T^n))}^{1-\theta}.
\end{align*}
Absorbing the fixed factors into the constant \(C\) proves
\eqref{eq:one-time-multiplicative}. The exponent \(\theta\) comes from
\Cref{prop:hilbert-propagation} and therefore depends only on \(n\), \(R\),
and \(|G|\).

We now derive the additive estimate. Young's inequality with conjugate
exponents \(1/(1-\theta)\) and \(1/\theta\) gives, for every \(\varepsilon>0\),
\begin{align*}
&C\exp\biggl(\frac{C}{s-t}\biggr)
\|z(t)\|_{L^2_{\mathcal F_t}(\Omega;L^2(G))}^\theta
\|z(s)\|_{L^2_{\mathcal F_s}(\Omega;L^2(\mathbb T^n))}^{1-\theta}\\
&\leq \varepsilon
\|z(s)\|_{L^2_{\mathcal F_s}(\Omega;L^2(\mathbb T^n))}
+C\varepsilon^{-\frac{1-\theta}{\theta}}
\exp\biggl(\frac{C}{\theta(s-t)}\biggr)
\|z(t)\|_{L^2_{\mathcal F_t}(\Omega;L^2(G))}.
\end{align*}
Combining this with \eqref{eq:one-time-multiplicative} and absorbing
\(1/\theta\) into the constant in the exponential proves
\eqref{eq:one-time-epsilon}.
\end{proof}

\subsection{Energy propagation and observability}

To pass from one-time interpolation to observability, we use the following
energy bound.

\begin{lemma}[One-sided energy propagation]\label{lem:energy-propagation}
Let $(z,Z)$ solve \eqref{eq:backward}.  Then, for
$0\leq t_1<t_2\leq T$,
\begin{align}\label{eq:energy-propagation}
\mathbb{E}\lVert z(t_1)\rVert_{L^2(\mathbb{T}^n)}^2
\leq e^{C_1^2}\mathbb{E}\lVert z(t_2)\rVert_{L^2(\mathbb{T}^n)}^2.
\end{align}
\end{lemma}

\begin{proof}
The following It\^o identities are first obtained for $(z_N,Z_N)$ with $s=t_2$
and then passed to the limit using \Cref{lem:galerkin-identification} and
\eqref{eq:coefficient-budget}. In particular, It\^o's formula
and periodic integration by parts yield
\begin{align*}
d\mathbb E\|z(r)\|_{L^2(\mathbb T^n)}^2
&=\Bigl[2\mathbb E\|\nabla z(r)\|_{L^2(\mathbb T^n)}^2
+\mathbb E\|Z(r)\|_{L^2(\mathbb T^n)}^2
-2\mathbb E\operatorname{Re}\langle b(r)Z(r),z(r)\rangle_{L^2(\mathbb T^n)}\Bigr]\,dr.
\end{align*}
The two terms involving $Z$ can be rewritten as
\begin{align*}
&\mathbb E\|Z(r)\|_{L^2(\mathbb T^n)}^2
-2\mathbb E\operatorname{Re}\langle b(r)Z(r),z(r)\rangle_{L^2(\mathbb T^n)}\\
&=\mathbb E\|Z(r)\|_{L^2(\mathbb T^n)}^2
-2\mathbb E\operatorname{Re}\langle Z(r),\overline{b(r)}z(r)\rangle_{L^2(\mathbb T^n)}
+\mathbb E\|\overline{b(r)}z(r)\|_{L^2(\mathbb T^n)}^2
-\mathbb E\|b(r)z(r)\|_{L^2(\mathbb T^n)}^2\\
&=\mathbb E\|Z(r)-\overline{b(r)}z(r)\|_{L^2(\mathbb T^n)}^2
-\mathbb E\|b(r)z(r)\|_{L^2(\mathbb T^n)}^2.
\end{align*}

To control the negative term, set
\begin{align*}
\Gamma(r)=\exp\Bigl(\int_{t_1}^r\beta(\omega,\tau)^2\,d\tau\Bigr),
\qquad t_1\le r\le t_2.
\end{align*}
Then $\Gamma(t_1)=1$, $\Gamma'=\beta^2\Gamma$ almost everywhere, and
\eqref{eq:coefficient-budget} implies
\begin{align*}
1\le\Gamma(r)\le e^{C_1^2},
\qquad t_1\le r\le t_2,
\quad\text{almost surely}.
\end{align*}
Applying the product formula to $\Gamma(r)\|z(r)\|_{L^2(\mathbb T^n)}^2$
and integrating from $t_1$ to $t_2$, we obtain, after taking expectations,
\begin{align*}
&\mathbb E\bigl[\Gamma(t_2)\|z(t_2)\|_{L^2(\mathbb T^n)}^2\bigr]
-\mathbb E\|z(t_1)\|_{L^2(\mathbb T^n)}^2\\
&=\mathbb E\int_{t_1}^{t_2}\Gamma(r)\Bigl[
2\|\nabla z(r)\|_{L^2(\mathbb T^n)}^2
+\|Z(r)-\overline{b(r)}z(r)\|_{L^2(\mathbb T^n)}^2
+\beta(r)^2\|z(r)\|_{L^2(\mathbb T^n)}^2
-\|b(r)z(r)\|_{L^2(\mathbb T^n)}^2\Bigr]\,dr.
\end{align*}
Here the stochastic integral has zero expectation by the same integrability argument
as in the proof of \Cref{prop:finite-analytic-energy}.
By \eqref{eq:real-axis-bound},
\begin{align*}
\beta(r)^2\|z(r)\|_{L^2(\mathbb T^n)}^2
-\|b(r)z(r)\|_{L^2(\mathbb T^n)}^2\ge0
\qquad\text{almost surely},
\end{align*}
so the entire drift integrand is nonnegative.
Using the bound on $\Gamma$, we obtain
\begin{align*}
\mathbb E\|z(t_1)\|_{L^2(\mathbb T^n)}^2
&\le\mathbb E\bigl[\Gamma(t_2)\|z(t_2)\|_{L^2(\mathbb T^n)}^2\bigr]
\le e^{C_1^2}\mathbb E\|z(t_2)\|_{L^2(\mathbb T^n)}^2,
\end{align*}
which proves \eqref{eq:energy-propagation}.
\end{proof}
\begin{proof}[Proof of \Cref{thm:state-observability}]
We follow the telescoping argument in \cite[proof of Theorem~1]{Apraiz2014}.

\textbf{Step 1.} We derive an estimate on each time interval.
With $\theta$ from \Cref{cor:one-time-interpolation}, set
\begin{align*}
\alpha&=\frac{1-\theta}{\theta},\qquad \rho_0=1-\frac\theta2,
\qquad t_k=T(1-\rho_0^k),\\
\delta_k&=t_{k+1}-t_k=T(1-\rho_0)\rho_0^k,
\qquad I_k=[t_k,t_k+\delta_k/2],\qquad k\geq0.
\end{align*}
Then $t_0=0$, $t_k\uparrow T$, and $\delta_{k+1}=\rho_0\delta_k$.
For $t\in I_k$, we have $t_{k+1}-t\geq\delta_k/2$, while
\Cref{lem:energy-propagation} gives
\begin{align*}
\lVert z(t_k)\rVert_{L^2_{\mathcal F_{t_k}}(\Omega;L^2(\mathbb T^n))}
\leq e^{C_1^2/2}
\lVert z(t)\rVert_{L^2_{\mathcal F_t}(\Omega;L^2(\mathbb T^n))}.
\end{align*}
Apply \eqref{eq:one-time-epsilon} with $s=t_{k+1}$ and with
$\varepsilon e^{-C_1^2/2}$ in place of $\varepsilon$.
Combining the two estimates gives
\begin{align*}
&\lVert z(t_k)\rVert_{L^2_{\mathcal F_{t_k}}(\Omega;L^2(\mathbb T^n))}
\leq\varepsilon
\lVert z(t_{k+1})\rVert_{L^2_{\mathcal F_{t_{k+1}}}(\Omega;L^2(\mathbb T^n))}
+C\varepsilon^{-\alpha}e^{2C/\delta_k}
\lVert z(t)\rVert_{L^2_{\mathcal F_t}(\Omega;L^2(G))},
\qquad t\in I_k.
\end{align*}
The first two norms are independent of $t$.  Integrating over $I_k$
and dividing by $|I_k|=\delta_k/2$, we obtain
\begin{align*}
&\lVert z(t_k)\rVert_{L^2_{\mathcal F_{t_k}}(\Omega;L^2(\mathbb T^n))}
\leq\varepsilon
\lVert z(t_{k+1})\rVert_{L^2_{\mathcal F_{t_{k+1}}}(\Omega;L^2(\mathbb T^n))}
+\frac{2C}{\delta_k}\varepsilon^{-\alpha}e^{2C/\delta_k}
\int_{I_k}\lVert z(t)\rVert_{L^2_{\mathcal F_t}(\Omega;L^2(G))}\,dt.
\end{align*}
Since $\delta_k^{-1}\leq e^{1/\delta_k}$, the factor from averaging
can be absorbed into the exponential.  Hence there are constants
$C,C_0>0$, depending only on $n,T,R,C_1,|G|$, such that
\begin{align}\label{eq:averaged-one-step}
&\lVert z(t_k)\rVert_{L^2_{\mathcal F_{t_k}}(\Omega;L^2(\mathbb T^n))}
\leq\varepsilon
\lVert z(t_{k+1})\rVert_{L^2_{\mathcal F_{t_{k+1}}}(\Omega;L^2(\mathbb T^n))}
+C\varepsilon^{-\alpha}e^{C_0/\delta_k}
\int_{I_k}\lVert z(t)\rVert_{L^2_{\mathcal F_t}(\Omega;L^2(G))}\,dt,
\qquad\varepsilon>0.
\end{align}
We keep $C_0$ fixed in the remaining argument.

\textbf{Step 2.} We choose weights that make the estimate telescope.
Set
\begin{align*}
C_2=(2-\theta)(C_0+1),\qquad
w_k=e^{-C_2/\delta_k},\qquad \varepsilon_k=\frac{w_{k+1}}{w_k}.
\end{align*}
The relation $\delta_{k+1}=\rho_0\delta_k$ gives
\begin{align*}
\varepsilon_k
=\exp\biggl(-\frac{C_2(\rho_0^{-1}-1)}{\delta_k}\biggr),
\qquad
1-\alpha(\rho_0^{-1}-1)
=1-\frac{1-\theta}{2-\theta}
=\frac1{2-\theta}.
\end{align*}
Consequently, the coefficient of the observation term satisfies
\begin{align*}
w_k\varepsilon_k^{-\alpha}e^{C_0/\delta_k}
&=\exp\biggl(
\frac{C_0-C_2+\alpha C_2(\rho_0^{-1}-1)}{\delta_k}\biggr)
=\exp\biggl(\frac{C_0-C_2/(2-\theta)}{\delta_k}\biggr)
=e^{-1/\delta_k}\leq1.
\end{align*}
Substitute $\varepsilon=\varepsilon_k$ into
\eqref{eq:averaged-one-step} and multiply by $w_k$.
Since $w_k\varepsilon_k=w_{k+1}$, we obtain
\begin{align*}
&w_k\lVert z(t_k)\rVert_{L^2_{\mathcal F_{t_k}}(\Omega;L^2(\mathbb T^n))}
-w_{k+1}\lVert z(t_{k+1})\rVert_{L^2_{\mathcal F_{t_{k+1}}}(\Omega;L^2(\mathbb T^n))} \\
&\leq Ce^{-1/\delta_k}
\int_{I_k}\lVert z(t)\rVert_{L^2_{\mathcal F_t}(\Omega;L^2(G))}\,dt \leq C\int_{I_k}
\lVert z(t)\rVert_{L^2_{\mathcal F_t}(\Omega;L^2(G))}\,dt.
\end{align*}

\textbf{Step 3.} We sum the estimates and pass to the limit.
Summing over $k=0,\ldots,m$ cancels all the intermediate weighted norms.
Since the intervals $I_k$ are disjoint subsets of $[0,T]$, this gives
\begin{align*}
&w_0\lVert z(0)\rVert_{L^2(\mathbb T^n)}
-w_{m+1}\lVert z(t_{m+1})\rVert_{L^2_{\mathcal F_{t_{m+1}}}(\Omega;L^2(\mathbb T^n))}\\
&\leq C\sum_{k=0}^m\int_{I_k}
\lVert  z(t)\rVert_{L^2_{\mathcal F_t}(\Omega;L^2(G))}\,dt \leq C\int_0^T
\lVert  z(t)\rVert_{L^2_{\mathcal F_t}(\Omega;L^2(G))}\,dt 
=C\lVert  z\rVert_{L^1_{\mathbb F}(0,T;L^2(\Omega;L^2(G)))}.
\end{align*}
By \eqref{eq:backward-solution-estimate},
\begin{align*}
0&\leq w_{m+1}
\lVert z(t_{m+1})\rVert_{L^2_{\mathcal F_{t_{m+1}}}(\Omega;L^2(\mathbb T^n))}
\leq C e^{-C_2/\delta_{m+1}}
\lVert\eta\rVert_{L^2_{\mathcal F_T}(\Omega;L^2(\mathbb T^n))}
\longrightarrow0
\qquad\text{as }m\to\infty,
\end{align*}
because $\delta_{m+1}\to0$.
Moreover, $\mathcal F_0$ is trivial up to null sets, so the norm of $z(0)$
above equals its $L^2(\mathbb T^n)$ norm.
Letting $m\to\infty$ and dividing by $w_0$ therefore yields
\begin{align}\label{eq:L1-observability-unsquared}
\lVert z(0)\rVert_{L^2(\mathbb T^n)}
\leq Cw_0^{-1}
\lVert   z\rVert_{L^1_{\mathbb F}(0,T;L^2(\Omega;L^2(G)))}.
\end{align}
Finally, $\delta_0=T\theta/2$ gives
$w_0^{-1}=\exp(2C_2/(T\theta))$.
Setting $C_{\mathrm{obs}}=Cw_0^{-1}$ proves \eqref{eq:state-observability}.
\end{proof}

\section{Null and approximate controllability by duality}\label{sec:duality}

We work over the real scalar field.  For each
$\eta\in L^2_{\mathcal F_T}(\Omega;L^2(\mathbb T^n))$, let
$(z^\eta,Z^\eta)$ solve \eqref{eq:backward} with terminal value $\eta$.
For every solution $y$ of \eqref{eq:forward}, we apply It\^o's formula to $(P_Ny,P_Nz^\eta)$.
By \eqref{eq:forward-solution-estimate}, \eqref{eq:backward-solution-estimate}, and
\eqref{eq:coefficient-budget}, the stochastic integrands belong to
$L^1_{\mathbb F}(\Omega;L^2(0,T))$, and the drift and endpoint terms converge as $N\to\infty$.
Taking expectations as above and passing to the limit gives the standard duality identity
\cite[Section~4.4]{Lue2021a}:
\begin{align}\label{eq:duality-identity}
\mathbb E\langle y(T),\eta\rangle_{L^2(\mathbb T^n)}
-\langle y_0,z^\eta(0)\rangle_{L^2(\mathbb T^n)}
=\mathbb E\int_0^T\langle u(t),z^\eta(t)\rangle_{L^2(G)}\,dt.
\end{align}
Here the coefficient terms cancel because
$\langle by,Z^\eta\rangle_{L^2(\mathbb T^n)}=\langle y,bZ^\eta\rangle_{L^2(\mathbb T^n)}$.

\begin{proof}[Proof of \Cref{thm:null-controllability}]
Fix $y_0\in L^2(\mathbb T^n)$.  Define the linear functional on the subspace
\begin{align*}
\mathcal X
=\{\mathbf{1}_G z^\eta:\eta\in L^2_{\mathcal F_T}(\Omega;L^2(\mathbb T^n))\}
\subset L^1_{\mathbb F}(0,T;L^2(\Omega;L^2(G))),
\end{align*}
as
\begin{align*}
\mathcal L(\mathbf{1}_G z^\eta)=-\langle y_0,z^\eta(0)\rangle_{L^2(\mathbb T^n)},\qquad \eta\in L^2_{\mathcal F_T}(\Omega;L^2(\mathbb T^n)).
\end{align*}
By \Cref{thm:state-observability}, we have
\begin{align*}
|\mathcal L(\mathbf{1}_G z^\eta)| = |\langle y_0,z^\eta(0)\rangle_{L^2(\mathbb T^n)}| \leq C_{\mathrm{obs}}\|y_0\|_{L^2(\mathbb T^n)} \|\mathbf{1}_G z^\eta \|_{L^1_{\mathbb F}(0,T;L^2(\Omega;L^2(G)))}, \qquad \forall\eta\in L^2_{\mathcal F_T}(\Omega;L^2(\mathbb T^n)).
\end{align*}
Hence,
$\mathcal L$ is a bounded linear functional with the norm
$\|\mathcal L\|\leq C_{\mathrm{obs}}\|y_0\|_{L^2(\mathbb T^n)}$.
The Hahn--Banach theorem extends $\mathcal L$ to the  
space $L^1_{\mathbb F}(0,T;L^2(\Omega;L^2(G)))$ without increasing its norm.  The representation theorem
\cite[Theorem~2.73]{Lue2021a}, with $p=1$, $q=2$, and $H=L^2(G;\mathbb R)$,
then gives a real
$u\in L^\infty_{\mathbb F}(0,T;L^2(\Omega;L^2(G)))$
satisfying \eqref{eq:control-cost} and
\begin{align*}
\mathbb E\int_0^T\langle u(t),z^\eta(t)\rangle_{L^2(G)}\,dt
=-\langle y_0,z^\eta(0)\rangle_{L^2(\mathbb T^n)},
\qquad\forall\eta\in L^2_{\mathcal F_T}(\Omega;L^2(\mathbb T^n)).
\end{align*}
This, together with \eqref{eq:duality-identity}, gives
\begin{align*}
	\mathbb E\langle y(T),\eta\rangle_{L^2(\mathbb T^n)}=0,
\qquad\forall\eta\in L^2_{\mathcal F_T}(\Omega;L^2(\mathbb T^n)).
\end{align*}
Taking $\eta=y(T)$ yields $y(T)=0$ almost surely.
\end{proof}

\begin{proof}[Proof of \Cref{thm:approximate-controllability}]
By linearity, it suffices to consider $y_0=0$.
Let $\eta\in L^2_{\mathcal F_T}(\Omega;L^2(\mathbb T^n))$ be orthogonal
to every terminal state reached by controls in
$L^\infty_{\mathbb F}(0,T;L^2(\Omega;L^2(G)))$.
By \eqref{eq:duality-identity},
\begin{align}\label{9.22-eq1}
\mathbb E\int_0^T\langle u(t),z^\eta(t)\rangle_{L^2(G)}\,dt=0, \qquad \forall u\in L^\infty_{\mathbb F}(0,T;L^2(\Omega;L^2(G))).
\end{align}
Taking $u=\mathbf{1}_G z^\eta$ in \eqref{9.22-eq1} gives
\begin{align*}
\|z^\eta(t)\|_{L^2_{\mathcal F_t}(\Omega;L^2(G))}=0
\qquad\text{for almost every }t<T.
\end{align*}
The interpolation estimate \eqref{eq:one-time-multiplicative}, with $s=T$,
implies that 
\begin{align}\label{9.22-eq2}
	\|z^\eta(t)\|_{L^2_{\mathcal F_t}(\Omega;L^2(\mathbb T^n))}=0
	\qquad\text{for almost every }t<T.
\end{align} 
Since $z^\eta(\cdot)\in L^2_{\mathbb F}(\Omega;C([0,T];L^2(\mathbb{T}^n)))$, we get from \eqref{9.22-eq2} that $\eta=z^\eta(T)=0$, $\mathbb P$-a.s.
Thus the reachable set has trivial orthogonal complement and is dense in
$L^2_{\mathcal F_T}(\Omega;L^2(\mathbb T^n))$, which proves the claim.
\end{proof}

\section{The analytic coefficient class: criterion and examples}\label{sec:richness}

We give a Fourier criterion for \Cref{ass:analytic-coefficient} and use
it to construct coefficients with arbitrarily high spatial frequencies
and infinitely many nonzero Fourier modes.

\begin{proposition}[Fourier realization of the analytic class]
\label{prop:richness-fourier}
Fix \(R>0\). Let \(\{a_{\mathbf k}\}_{\mathbf k\in\mathbb Z^n}\) be a family of complex-valued progressively measurable processes satisfying
\begin{align}\label{eq:fourier-conjugate-symmetry}
a_{-\mathbf k}(\omega,t)=\overline{a_{\mathbf k}(\omega,t)},
\qquad \mathbf k\in\mathbb Z^n,
\end{align}
and define
\begin{align}\label{eq:richness-majorant}
\mathcal C(\omega,t)
:=\sum_{\mathbf k\in\mathbb Z^n}|a_{\mathbf k}(\omega,t)|e^{R|\mathbf k|_1}.
\end{align}
Suppose that, for some \(C>0\),
\begin{align}\label{eq:richness-fourier-budget}
\|\mathcal C\|_{L^\infty_{\mathbb F}(\Omega;L^2(0,T))}^2\le C^2.
\end{align}
Then, after possibly modifying the \(a_{\mathbf k}\) on a common progressive \((\mathbb P\otimes dt)\)-null set, the series
\begin{align}\label{eq:richness-fourier-series}
\widetilde b(\omega,t,\zeta)
:=\sum_{\mathbf k\in\mathbb Z^n}a_{\mathbf k}(\omega,t)e^{\mathrm{i}\mathbf k\cdot \zeta}
\end{align}
admits a real trace \(b\) that satisfies \Cref{ass:analytic-coefficient} and
\begin{align}\label{eq:richness-beta-bound}
\|b(\omega,t,\cdot)\|_{\mathcal A_R(\mathbb T^n)}
\le \mathcal C(\omega,t)
\end{align}
for every \((\omega,t)\in\Omega\times(0,T)\).
\end{proposition}

\begin{proof}
The finite sums in \eqref{eq:richness-majorant} increase to \(\mathcal C(\cdot,\cdot)\), so \(\mathcal C(\cdot,\cdot)\) is progressively measurable. By \eqref{eq:richness-fourier-budget}, the set \(\{\mathcal C=\infty\}\) is a progressive \((\mathbb P\otimes dt)\)-null set. We set each \(a_{\mathbf k}\) equal to zero on this common set. Conjugate symmetry is preserved, and the modified weighted sum, still denoted by \(\mathcal C(\cdot,\cdot)\), is finite everywhere and still satisfies \eqref{eq:richness-fourier-budget}.

\textbf{Step 1.} We first establish uniform convergence on the polystrip. Fix \((\omega,t)\) and define
\begin{align*}
\widetilde b_N(\omega,t,\zeta)
=\sum_{|\mathbf k|_\infty\le N}a_{\mathbf k}(\omega,t)e^{\mathrm{i}\mathbf k\cdot \zeta},
\qquad N\in\mathbb N.
\end{align*}
For \(\zeta=x+\mathrm{i}y\in S_R^n\),
\begin{align*}
|e^{\mathrm{i}\mathbf k\cdot \zeta}|
=e^{-\mathbf k\cdot y}
\le e^{\sum_{j=1}^n|k_j||y_j|}
\le e^{R|\mathbf k|_1}.
\end{align*}
Consequently, for \(M>N\),
\begin{align*}
\sup_{\zeta\in S_R^n}
|\widetilde b_M(\omega,t,\zeta)-\widetilde b_N(\omega,t,\zeta)|
&\le\sum_{N<|\mathbf k|_\infty\le M}
|a_{\mathbf k}(\omega,t)|e^{R|\mathbf k|_1}\\
&\le\sum_{|\mathbf k|_\infty>N}
|a_{\mathbf k}(\omega,t)|e^{R|\mathbf k|_1}
\to 0
\qquad\text{as }N\to\infty.
\end{align*}
The last limit follows from \(\mathcal C(\omega,t)<\infty\) and is uniform in \(M>N\). Thus \eqref{eq:richness-fourier-series} converges absolutely and uniformly on \(S_R^n\) for each fixed \((\omega,t)\). Each partial sum is holomorphic and \(2\pi\)-periodic in each coordinate, so these properties pass to the limit. Moreover,
\begin{align*}
\sup_{\zeta\in S_R^n}|\widetilde b(\omega,t,\zeta)|
\le\sum_{\mathbf k\in\mathbb Z^n}|a_{\mathbf k}(\omega,t)|e^{R|\mathbf k|_1}
=\mathcal C(\omega,t).
\end{align*}

\textbf{Step 2.} We now verify the coefficient assumptions. For \(x\in\mathbb T^n\), conjugate symmetry and the change of index \(\mathbf k\mapsto-\mathbf k\) yield
\begin{align*}
\overline{\widetilde b_N(\omega,t,x)}
&=\sum_{|\mathbf k|_\infty\le N}
\overline{a_{\mathbf k}(\omega,t)}e^{-\mathrm{i}\mathbf k\cdot x}
=\sum_{|\mathbf k|_\infty\le N}
a_{-\mathbf k}(\omega,t)e^{-\mathrm{i}\mathbf k\cdot x}
=\widetilde b_N(\omega,t,x).
\end{align*}
Passing to the limit shows that the real trace \(b(\omega,t,x)=\widetilde b(\omega,t,x)\) is real-valued. Each \(\widetilde b_N\) is \(\mathbb F\otimes\mathcal B(S_R^n)\)-measurable, since its coefficients are progressively measurable and its spatial factors are continuous. The pointwise limit \(\widetilde b\) has the same measurability, and its restriction \(b\) to the real torus is progressively measurable.

We have therefore shown that \(b(\omega,t,\cdot)\in\mathcal A_R(\mathbb T^n)\) for every \((\omega,t)\), and that
\begin{align*}
\beta(\omega,t)
=\|b(\omega,t,\cdot)\|_{\mathcal A_R(\mathbb T^n)}
=\sup_{\zeta\in S_R^n}|\widetilde b(\omega,t,\zeta)|
\le \mathcal C(\omega,t).
\end{align*}
By continuity in \(\zeta\), this supremum can be taken over a fixed countable dense subset of \(S_R^n\), so \(\beta\) is progressively measurable. Finally, \eqref{eq:richness-fourier-budget} yields
\begin{align*}
\|\beta\|_{L^\infty_{\mathbb F}(\Omega;L^2(0,T))}
\le\|\mathcal C\|_{L^\infty_{\mathbb F}(\Omega;L^2(0,T))}
\le C.
\end{align*}
This proves \eqref{eq:richness-beta-bound} and all the requirements of \Cref{ass:analytic-coefficient}.
\end{proof}

By choosing finitely many nonzero coefficients in \Cref{prop:richness-fourier}, one obtains random trigonometric polynomials with arbitrary frequencies. A deterministic conjugate-symmetric sequence with infinitely many nonzero coefficients and finite weighted sum in \eqref{eq:richness-majorant} yields a profile that is not a trigonometric polynomial. Consequently, \Cref{ass:analytic-coefficient} encompasses an infinite-dimensional class of spatial profiles, whose multiplication operators generally mix Fourier modes.

\begin{example}[A random, time- and space-dependent coefficient]
Let \(M>0\) and define
\begin{align*}
\widetilde b(\omega,t,\zeta)
={}&\frac M2\tanh(W(t))
+\frac{M}{2n\cosh R}
\tanh\Bigl(\int_0^tW(r)\,dr\Bigr)
\sum_{j=1}^n\cos\zeta_j.
\end{align*}
The random amplitudes are bounded and progressively measurable, which ensures the required joint measurability of \(\widetilde b\).
Since \(|\cos\zeta_j|\le\cosh R\) on \(S_R^n\), the real trace \(b(\omega,t,\cdot)\) of \(\widetilde b(\omega,t,\cdot)\) belongs to \(\mathcal A_R(\mathbb T^n)\) and satisfies
$
\|b(\omega,t,\cdot)\|_{\mathcal A_R(\mathbb T^n)}\le M$. 
Hence \Cref{ass:analytic-coefficient} holds with \(C_1=M\sqrt T\).
\end{example}

\begin{example}[A random coefficient with infinitely many Fourier modes]
Let $M>0$. For $\zeta\in S_R^n$, set
\begin{align*}
F_R(\zeta)
&=\frac1n\sum_{j=1}^n\sum_{m=1}^{\infty}
2^{-m}e^{-mR}\cos(m\zeta_j),
\qquad
A(t)=\frac{2+\tanh\bigl(\int_0^tW(r)\,dr\bigr)}3.
\end{align*}
We define the coefficient by
\begin{align*}
\widetilde b(\omega,t,\zeta)
=\frac M2\tanh(W(t))+\frac M2 A(t)F_R(\zeta).
\end{align*}
The random amplitudes are bounded and progressively measurable, with $1/3<A(t)<1$.
Since $2^{-m}e^{-mR}|\cos(m\zeta_j)|\le 2^{-m}$ on $S_R^n$, the series defining $F_R$ converges absolutely and uniformly to a periodic holomorphic function with real trace and supremum norm at most $1$.
The partial sums ensure the required joint measurability of $\widetilde b$.
Thus the real trace $b(\omega,t,\cdot)$ belongs to $\mathcal A_R(\mathbb T^n)$ and satisfies
$
\|b(\omega,t,\cdot)\|_{\mathcal A_R(\mathbb T^n)}
\le\frac M2|\tanh(W(t))|+\frac M2 A(t)\le M$.
Hence \Cref{ass:analytic-coefficient} holds with $C_1=M\sqrt T$.
Since $A(t)>1/3$, the Fourier coefficients at $\pm m\mathbf e_j$ are nonzero for every $m\ge1$ and $1\le j\le n$.
\end{example}

The analytic norm may also be unbounded in time for a profile with
infinitely many nonzero Fourier modes. With $F_R$ as above, set
\begin{align*}
b(t,x)=M(T-t)^{-1/4}F_R(x),
\qquad 0<t<T.
\end{align*}
Then $b(t,\cdot)\in\mathcal A_R(\mathbb T^n)$ and
\begin{align*}
\beta(t)\le M(T-t)^{-1/4},
\qquad
\int_0^T\beta(t)^2\,dt\le2M^2\sqrt T.
\end{align*}
Hence \Cref{ass:analytic-coefficient} holds with
$C_1^2=2M^2\sqrt T$. Since $F_R(0)>0$, this coefficient is genuinely
unbounded as $t\uparrow T$.

\section{Concluding remarks}

In this paper we established null and approximate controllability for a stochastic parabolic equation with a space-dependent analytic noise coefficient, together with an observability inequality with an \(L^1\)-in-time mean-square observation norm. The examples in Section \ref{sec:richness} show that \Cref{ass:analytic-coefficient} accommodates random trigonometric polynomials, profiles with infinitely many nonzero Fourier modes, and coefficients that are unbounded in time. These coefficients must nevertheless admit bounded holomorphic extensions to a complex strip of fixed positive width, with the corresponding analytic norms satisfying a square-integrability bound in time that is uniform over sample paths. Having shown that the analytic coefficient class is rich, we now discuss the limitations of the analyticity assumption and indicate several directions for further study.

The core of our proof follows the chain
\[
\begin{aligned}
&\text{analytic energy estimate (Theorem \ref{thm:analytic-smoothing})} \;\longrightarrow\; \text{propagation of smallness (Proposition \ref{prop:hilbert-propagation})} \\
&\longrightarrow\; \text{one-time interpolation (Corollary \ref{cor:one-time-interpolation})} \;\longrightarrow\; \text{observability (Theorem \ref{thm:state-observability})}.
\end{aligned}
\]
The key step that propagates information from the local observation set \(G\) to the whole torus \(\mathbb{T}^n\) is the analytic propagation of smallness used in Proposition \ref{prop:hilbert-propagation}. Its proof rests on the multidimensional analytic propagation theorem  \cite[Theorem~4]{Apraiz2014}, whose essence is that an analytic function on its domain of convergence is uniquely determined by its values on any subset of positive measure.

A natural question is whether the spatial analyticity condition can be relaxed to Gevrey regularity of order greater than $1$. The current proof cannot directly cover this case, and there are mainly two obstacles. First, general non-analytic Gevrey functions do not possess holomorphic extensions in a neighborhood, and therefore Gevrey regularity alone cannot guarantee the holomorphic extensions and multiplier identities required for the complex translation in this paper. Even if almost-analytic extensions are used, they cannot directly replace the holomorphic extensions employed in the proof. Second, Gevrey classes of order greater than 1 contain nonzero smooth functions supported in a local region, and these functions can vanish identically on a nonempty open set. Consequently, the propagation-of-smallness estimate of the same type given in Proposition \ref{prop:hilbert-propagation}, which controls the global norm from an arbitrary set of positive measure, cannot hold universally for the entire space of such Gevrey functions. This obstacle already exists in the scalar case and is not caused by the generalization to Hilbert space-valued functions.

\appendix

\section{Proof of \Cref{lem:galerkin-identification}}\label{app:galerkin}

The proof of \Cref{lem:galerkin-identification} follows a standard
Galerkin stability argument.  We include it here for completeness.

\begin{proof}[Proof of \Cref{lem:galerkin-identification}]
Fix \(s\in(0,T]\), and let \((z_N,Z_N)\) be the solution of \eqref{eq:galerkin-bsde} with \(\xi=z(s)\).
Set \(e_N=z_N-P_Nz\) and \(h_N=Z_N-P_NZ\). Subtracting the projection of \eqref{eq:backward} from \eqref{eq:galerkin-bsde} yields
\begin{equation*}
\begin{cases}
	\begin{aligned}
		&d e_N
		=(-\Delta e_N-P_N(bh_N)+f_N)\,dr+h_N\,dW(r)
		&&\text{in }(0,s),\\
		&e_N(s)=0,
	\end{aligned}
\end{cases}
\end{equation*}
where \(f_N=P_N\bigl(b(Z-P_NZ)\bigr)\).
By \eqref{eq:coefficient-budget} and the strong convergence of \(P_N\),
\begin{align}\label{eq:galerkin-residual}
\mathbb{E}\Bigl(\int_0^s\|f_N(r)\|_{L^2(\mathbb{T}^n)}\,dr\Bigr)^2
\leq C_1^2\mathbb{E}\int_0^s\|Z(r)-P_NZ(r)\|_{L^2(\mathbb{T}^n)}^2\,dr
\longrightarrow0.
\end{align}

A standard adaptation of the stability argument in
\cite[Section~4.4]{Lue2021a}, using
\eqref{eq:real-axis-bound} and \eqref{eq:coefficient-budget}, yields
\begin{align*}
\mathbb E\sup_{0\leq r\leq s}\|e_N(r)\|_{L^2(\mathbb{T}^n)}^2
\leq C\mathbb E\Bigl(\int_0^s\|f_N(r)\|_{L^2(\mathbb{T}^n)}\,dr\Bigr)^2,
\end{align*}
where, throughout this proof, \(C\) depends only on \(T\) and \(C_1\).
To estimate the integral norms, define
\(\gamma(r)=\exp\bigl(2\int_0^r\beta(\tau)^2\,d\tau\bigr)\).
It\^o's formula gives
\begin{align*}
d \bigl(\gamma\|e_N\|_{L^2(\mathbb{T}^n)}^2\bigr)
={}&\gamma\bigl(2\|\nabla e_N\|_{L^2(\mathbb{T}^n)}^2
+\|h_N\|_{L^2(\mathbb{T}^n)}^2
+2\beta^2\|e_N\|_{L^2(\mathbb{T}^n)}^2  -2\operatorname{Re}\langle bh_N,e_N\rangle_{L^2(\mathbb{T}^n)}\\
& \quad
+2\operatorname{Re}\langle f_N,e_N\rangle_{L^2(\mathbb{T}^n)}\bigr)\,dr
+2\gamma\operatorname{Re}\langle e_N,h_N\rangle_{L^2(\mathbb{T}^n)}\,dW(r).
\end{align*}
Since $1\leq\gamma\leq e^{2C_1^2}$, the same integrability argument as in the proof of
\Cref{prop:finite-analytic-energy} shows that the stochastic integral has zero expectation. Using
\begin{align*}
2|\langle bh_N,e_N\rangle_{L^2(\mathbb{T}^n)}|
\leq\tfrac12\|h_N\|_{L^2(\mathbb{T}^n)}^2
+2\beta^2\|e_N\|_{L^2(\mathbb{T}^n)}^2,
\end{align*}
integrating from \(0\) to \(s\), and taking expectations, noting that \(e_N(s)=0\), we obtain
from the Cauchy--Schwarz inequality for the forcing term, together with the
preceding stability estimate and \(1\leq\gamma\leq e^{2C_1^2}\),
\begin{align}\label{eq:galerkin-error-estimate}
&\mathbb{E}\sup_{0\leq r\leq s}\|e_N(r)\|_{L^2(\mathbb{T}^n)}^2
+\mathbb{E}\int_0^s\bigl(\|\nabla e_N(r)\|_{L^2(\mathbb{T}^n)}^2
+\|h_N(r)\|_{L^2(\mathbb{T}^n)}^2\bigr)\,dr \leq C\mathbb{E}\Bigl(\int_0^s\|f_N(r)\|_{L^2(\mathbb{T}^n)}\,dr\Bigr)^2.
\end{align}
Combining \eqref{eq:galerkin-residual} and \eqref{eq:galerkin-error-estimate}
yields 
$$
\lim_{N\to\infty}e_N=0 \qquad\mbox{ in }L^2_{\mathbb F}(\Omega;C([0,s];L^2(\mathbb{T}^n)))
\cap L^2_{\mathbb F}(0,s;H^1(\mathbb{T}^n))
$$
and
$$
\lim_{N\to\infty}h_N=0  \qquad\mbox{ in }  L^2_{\mathbb F}(0,s;L^2(\mathbb{T}^n)).
$$
Consequently, \eqref{eq:galerkin-strong-time-1} holds.
\end{proof}

\bibliographystyle{siam-paper-no}
\bibliography{references}

@Article{Tang2009,
  author     = {Tang, Shanjian and Zhang, Xu},
  journal    = {SIAM J. Control Optim.},
  title      = {Null controllability for forward and backward stochastic parabolic equations},
  year       = {2009},
  issn       = {0363-0129},
  number     = {4},
  pages      = {2191--2216},
  volume     = {48},
  doi        = {10.1137/050641508},
  fjournal   = {SIAM Journal on Control and Optimization},
  mrclass    = {93B05 (35R60 60H15 93C20 93E03)},
  mrnumber   = {2520325},
}

@Article{Yang2016a,
  author   = {Yang, Donghui and Zhong, Jie},
  journal  = {SIAM J. Control Optim.},
  title    = {Observability inequality of backward stochastic heat equations for measurable sets and its applications},
  year     = {2016},
  issn     = {0363-0129},
  number   = {3},
  pages    = {1157--1175},
  volume   = {54},
  doi      = {10.1137/15M1033289},
  fjournal = {SIAM Journal on Control and Optimization},
}

@Article{Apraiz2014,
  author   = {Apraiz, Jone and Escauriaza, Luis and Wang, Gengsheng and Zhang, C.},
  journal  = {J. Eur. Math. Soc. (JEMS)},
  title    = {Observability inequalities and measurable sets},
  year     = {2014},
  issn     = {1435-9855},
  number   = {11},
  pages    = {2433--2475},
  volume   = {16},
  doi      = {10.4171/JEMS/490},
  fjournal = {Journal of the European Mathematical Society},
}

@Article{Liu2026,
  author   = {Liu, Yuanhang and Wu, Weijia and Yang, Donghui and Zhong, Jie},
  title    = {Observability Inequalities for the Backward Stochastic Evolution Equations and Their Applications},
  journal  = {SIAM J. Control Optim.},
  year     = {2026},
  volume   = {64},
  number   = {2},
  pages    = {727--747},
  doi      = {10.1137/23M159576X},
  fjournal = {SIAM Journal on Control and Optimization},
}

@Article{Lu2011,
  author   = {L{\"u}, Qi},
  title    = {Some results on the controllability of forward stochastic heat equations with control on the drift},
  journal  = {J. Funct. Anal.},
  year     = {2011},
  volume   = {260},
  number   = {3},
  pages    = {832--851},
  doi      = {10.1016/j.jfa.2010.10.018},
  fjournal = {Journal of Functional Analysis},
}

@Article{HernandezSantamaria2023,
  author    = {Hern{\'a}ndez-Santamar{\'\i}a, V{\'\i}ctor and Le Balc'h, K{\'e}vin and Peralta, Liliana},
  title     = {Global null-controllability for stochastic semilinear parabolic equations},
  journal   = {Ann. Inst. H. Poincar{\'e} Anal. Non Lin{\'e}aire},
  year      = {2023},
  volume    = {40},
  number    = {6},
  pages     = {1415--1455},
  issn      = {1873-1430},
  doi       = {10.4171/AIHPC/69},
  fjournal  = {Annales de l'Institut Henri Poincar{\'e} C, Analyse non lin{\'e}aire},
}

@Article{Liu2014,
  author   = {Liu, Xu},
  title    = {Global {C}arleman estimate for stochastic parabolic equations, and its application},
  journal  = {ESAIM Control Optim. Calc. Var.},
  year     = {2014},
  volume   = {20},
  number   = {3},
  pages    = {823--839},
  issn     = {1292-8119},
  doi      = {10.1051/cocv/2013085},
  fjournal = {ESAIM. Control, Optimisation and Calculus of Variations},
}

@Book{Lue2021a,
  author    = {L{\"u}, Qi and Zhang, Xu},
  title     = {Mathematical Control Theory for Stochastic Partial Differential Equations},
  series    = {Probability Theory and Stochastic Modelling},
  volume    = {101},
  publisher = {Springer},
  address   = {Cham},
  year      = {2021},
  pages     = {xiii+592},
  isbn      = {978-3-030-82330-6},
  doi       = {10.1007/978-3-030-82331-3},
}

@Article{Barbu2003,
  author   = {Barbu, Viorel and R{\u{a}}{\c{s}}canu, Aurel and Tessitore, Gianmario},
  title    = {Carleman estimates and controllability of linear stochastic heat equations},
  journal  = {Appl. Math. Optim.},
  fjournal = {Applied Mathematics and Optimization},
  year     = {2003},
  volume   = {47},
  number   = {2},
  pages    = {97--120},
  issn     = {0095-4616},
  doi      = {10.1007/s00245-002-0757-z},
}

@Article{ZhangXuLiu2026,
  author   = {Zhang, Lei and Xu, Fan and Liu, Bin},
  title    = {New global {Carleman} estimates and null controllability for forward/backward semilinear parabolic {SPDEs}},
  journal  = {SIAM J. Control Optim.},
  fjournal = {SIAM Journal on Control and Optimization},
  year     = {2026},
  volume   = {64},
  number   = {3},
  pages    = {1297--1326},
  doi      = {10.1137/25M1770643},
}

@InCollection{Zuazua2007,
  author    = {Zuazua, Enrique},
  title     = {Controllability and observability of partial differential equations: some results and open problems},
  booktitle = {Handbook of Differential Equations: Evolutionary Equations. {V}ol. {III}},
  publisher = {Elsevier/North-Holland},
  address   = {Amsterdam},
  year      = {2007},
  pages     = {527--621},
  doi       = {10.1016/S1874-5717(07)80010-7},
}

@InCollection{FernandezCara2012,
  author    = {Fern{\'a}ndez-Cara, Enrique},
  title     = {The control of {PDEs}: some basic concepts, recent results and open problems},
  booktitle = {Topics in Mathematical Modeling and Analysis},
  editor    = {Kaplick{\'y}, Petr},
  series    = {Jind{\v r}ich Ne{\v c}as Cent. Math. Model. Lect. Notes},
  volume    = {7},
  publisher = {Matfyzpress},
  address   = {Prague},
  year      = {2012},
  pages     = {49--107},
  isbn      = {978-80-7378-196-5},
  url       = {https://www.karlin.mff.cuni.cz/~prusv/ncmm/notes/download/volume-vii.pdf},
}

@Book{LeRousseauLebeauRobbiano2022,
  author    = {{Le Rousseau}, J{\'e}r{\^o}me and Lebeau, Gilles and Robbiano, Luc},
  title     = {Elliptic {Carleman} Estimates and Applications to Stabilization and Controllability, Volume {I}: {Dirichlet} Boundary Conditions on {Euclidean} Space},
  series    = {Progr. Nonlinear Differential Equations Appl.},
  volume    = {97},
  publisher = {Birkh{\"a}user},
  address   = {Cham},
  year      = {2022},
  pages     = {viii+411},
  isbn      = {978-3-030-88673-8},
  doi       = {10.1007/978-3-030-88674-5},
  note      = {{PNLDE} Subseries in Control},
}

\end{document}